\documentclass[reqno, a4]{amsart}

\usepackage{amsfonts, amsthm, amssymb, amsmath, mathrsfs}
\usepackage[all]{xy}
\usepackage{graphicx}
\usepackage[usenames,dvipsnames]{color}
\usepackage[backref=page]{hyperref}
\usepackage{enumitem}
\hypersetup{
 colorlinks,
 citecolor=Green,
 linkcolor=blue,
 urlcolor=Blue}

\newtheorem{thm}{Theorem}[section]
\newtheorem{prop}[thm]{Proposition}
\newtheorem{lem}[thm]{Lemma}
\newtheorem{cor}[thm]{Corollary}

\theoremstyle{definition}

\newtheorem{rmk}[thm]{Remark}

\numberwithin{equation}{section}

\newcommand{\PP}{\ensuremath{\mathbb{P}}}
\newcommand{\GG}{\ensuremath{\mathbb{G}}}
\def\Sym{\operatorname{Sym}}

\begin{document}

\title[Irreducibility and components rigid in moduli of the Hilbert scheme]{Irreducibility and components rigid in moduli of \\ the Hilbert scheme of smooth curves}

\author[C. Keem, Y.H. Kim and A.F. Lopez]{Changho Keem*, Yun-Hwan Kim and Angelo Felice Lopez**}

\thanks{* Supported in part by NRF Grant \# 2011-0010298.}

\thanks{** Supported in part by the MIUR national project ``Geometria delle variet\`a algebriche" PRIN 2010-2011.}

\address{\hskip -.43cm Department of Mathematics, Seoul National University, Seoul 151-742, South Korea. \newline e-mail {\tt ckeem1@gmail.com, yunttang@snu.ac.kr}}

\address{\hskip -.43cm Dipartimento di Matematica e Fisica, Universit\`a di Roma
Tre, Largo San Leonardo Murialdo 1, 00146, Roma, Italy. \newline e-mail {\tt lopez@mat.uniroma3.it}}



\date{}

\begin{abstract} Denote by $\mathcal{H}_{d,g,r}$ the Hilbert scheme of smooth curves, that is the union of components whose general point corresponds to a smooth irreducible and non-degenerate curve of degree $d$ and genus $g$ in $\PP^r$. A component of $\mathcal{H}_{d,g,r}$ is rigid in moduli if its image under the natural map $\pi:\mathcal{H}_{d,g,r} \dashrightarrow \mathcal{M}_{g}$ is a one point set. In this note, we provide a proof of the fact that $\mathcal{H}_{d,g,r}$ has no components rigid in moduli for $g > 0$ and $r=3$, from which it follows that the only  
smooth projective curves embedded in $\mathbb{P}^3$ whose only deformations are given by projective transformations are the twisted cubic curves. In case $r \geq 4$, we also prove the non-existence of a component of $\mathcal{H}_{d,g,r}$ rigid in moduli in a certain restricted range of $d$, $g>0$ and $r$. In the course of the proofs, we establish  the irreducibility of $\mathcal{H}_{d,g,3}$ beyond the range which has been known  before. 
\end{abstract}

\maketitle

\section{Basic set up, terminologies and preliminary results}

Given non-negative integers $d$, $g$ and $r$, let $H_{d,g,r}$ be the Hilbert scheme parametri-zing curves of degree $d$ and genus $g$ in $\PP^r$ and let $\mathcal{H}_{d,g,r}$ be the Hilbert scheme of smooth curves, that is the union of components of $H_{d,g,r}$ whose general point corresponds to a smooth irreducible and non-degenerate curve of degree $d$ and genus $g$ in $\PP^r$. Let $\mathcal{M}_g$ be the moduli space of smooth curves of genus $g$ and consider the natural rational map
\[ \pi: \mathcal{H}_{d,g,r} \dashrightarrow \mathcal{M}_g \]
which sends each point $c\in  \mathcal{H}_{d,g,r}$ representing a smooth irreducible non-degenerate curve $C$ in $\PP^r$ to the corresponding isomorphism class $[C]\in\mathcal{M}_g$. 

In this article,  we concern ourselves with the question regarding the existence of an irreducible component  $\mathcal{Z}$ of $\mathcal{H}_{d,g,r}$ whose image under the map $\pi$ is just a one point set in $\mathcal{M}_g$, which we call a {\bf component rigid in moduli}. 

It is a folklore conjecture that such components should not exist, except when $g=0$. It is also expected \cite[1.47]{HM} that there are no rigid curves in $\PP^r$, that is curves that admit no deformations other than those given by projectivities of $\PP^r$, except for rational normal curves. 

\medskip
In the next two sections, we provide a proof of the fact that $\mathcal{H}_{g+1,g,3}$ is irreducible and $\mathcal{H}_{d,g,3}$ does not have a component rigid in moduli if $g>0$. This in turn implies that  there are no rigid curves in $\mathbb{P}^3$ except for twisted cubic curves. In the subsequent section we also prove that, for $r \ge 4$, $\mathcal{H}_{d,g,r}$ does not carry any component rigid in moduli in a certain restricted range with respect to $d,g>0$ and $r$.  In proving the results,  we utilize several classical theorems including the so-called Accola-Griffiths-Harris' bound on the dimension of a component consisting of birationally very ample linear series in the variety of special linear series on a smooth algebraic curve.  
We work over the field of complex numbers.

\vskip 4pt

For notation and conventions, we usually follow those in \cite{ACGH}; e.g. $\pi (d,r)$ is the maximal possible arithmetic genus of an irreducible and non-degenerate curve of degree $d$ in $\PP^r$.
Before proceeding, we recall several related results that are rather well known;  cf. \cite{AC1}. 

\vskip 4pt
For any given isomorphism class $[C] \in \mathcal{M}_g$ corresponding to a smooth irreducible curve $C$, there exist a neighborhood $U\subset \mathcal{M}_g$ of $[C]$ and a smooth connected variety $\mathcal{M}$ which is a finite ramified covering $h:\mathcal{M} \to U$, together with varieties $\mathcal{C}$, $\mathcal{W}^r_d$ and $\mathcal{G}^r_d$ which are proper over $\mathcal{M}$ with the following properties:
\begin{enumerate}
\item[(1)] $\xi:\mathcal{C}\to\mathcal{M}$ is a universal curve, that is for every $p\in \mathcal{M}$, $\xi^{-1}(p)$ is a smooth curve of genus $g$ whose isomorphism class is $h(p)$,
\item[(2)] $\mathcal{W}^r_d$ parametrizes  pairs $(p,L)$, where $L$ is a line bundle of degree $d$ with $h^0(L) \ge r+1$,
\item[(3)] $\mathcal{G}^r_d$ parametrizes couples $(p, \mathcal{D})$, where $\mathcal{D}$ is  possibly an incomplete linear series of degree $d$ and of dimension $r$, which is denoted by $g^r_d$, on $\xi^{-1}(p)$.
\end{enumerate}

Let $\mathcal{G}$ be the union of components of $\mathcal{G}^r_d$ whose general element $(p,\mathcal{D})$ corresponds to a very ample linear series $\mathcal{D}$ on the curve $C= {\xi}^{-1}(p)$. 
Note that the open subset of $\mathcal{H}_{d,g,r}$ consisting of points corresponding to smooth curves is a $\PP GL(r+1)$-bundle over an open subset of $\mathcal{G}$. 

\vskip 4pt
We also make a note of the following fact which is basic in the theory; cf. 
\cite{AC1} or \cite[Chapter 2]{Hr}.
\begin{prop}
\label{principal}
There exists a unique component $\mathcal{G}_0$ of $\mathcal{G}$ which dominates $\mathcal{M}$ (or $\mathcal{M}_g$) if the Brill-Noether number $\rho(d, g, r):=g-(r+1)(g-d+r)$ is non-negative. Furthermore in this case, for any possible component $\mathcal{G}'$ of $\mathcal{G}$ other than $\mathcal{G}_0$, a general element $(p, \mathcal{D})$ of $\mathcal{G}'$ is such that $\mathcal{D}$ is a special linear system on $C = \xi^{-1}(p)$.
\end{prop}

\begin{rmk}
\label{dom}
In the Brill-Noether range, that is $\rho(d,g,r)\ge0$, the unique component $\mathcal{G}_0$ of $\mathcal{G}$ (and the corresponding component $\mathcal{H}_0$ of $\mathcal{H}_{d,g,r}$ as well) which dominates $\mathcal{M}$ or $\mathcal{M}_g$ is called the  ``principal component".  We call the other possible components ``exceptional components".
\end{rmk}

We recall the following well-known fact on the dimension of a component of the Hilbert scheme $\mathcal{H}_{d,g,r}$; cf. \cite[Chapter 2.a]{Hr}  or \cite[1.E]{HM}. 

\begin{thm}
\label{expecteddim}
Let $c\in \mathcal{H}_{d,g,r}$ be a point representing a curve $C$ in $\PP^r$. The tangent space of $\mathcal{H}_{d,g,r}$ at $c$ can be identified as
\[ T_{c}\mathcal{H}_{d,g,r}=H^0\left(C,N_{C/\PP^r}\right), \]
where $N_{C/\PP^r}$ is the normal sheaf of $C$ in $\PP^r$. Moreover, if $C$ is a locally complete intersection, in particular if $C$ is smooth, then
\[ \chi(N_{C/\PP^r})\leq \dim_{c}\mathcal{H}_{d,g,r}\leq h^0\left(C,N_{C/\PP^r}\right), \]
where $\chi(N_{C/\PP^r})=h^0\left(C,N_{C/\PP^r}\right)-h^1\left(C,N_{C/\PP^r}\right)$.
\end{thm}

\vskip 6pt
For a locally complete intersection $c\in \mathcal{H}_{d,g,r}$, we have
\[ \chi(N_{C/\PP^r})=(r+1)d-(r-3)(g-1) \]
which is denoted by $\lambda(d,g,r)$. 

\vskip 6pt
The following bound on the dimension of the variety of special linear series on a fixed smooth algebraic curve shall become useful in subsequent sections.

\begin{thm}[Accola-Griffiths-Harris Theorem; \rm{\cite[p.73]{Hr}}
\label{agh}]
Let $C$ be a curve of genus $g$, $|D|$ a birationally very ample special $g^r_d$, that is a special linear system of dimension $r$ and degree $d$ inducing a birational morphism from $C$ onto a curve of degree $d$ in $\PP^r$.
Then in a neighborhood of $|D|$ on ${J}(C)$, either $\dim W^r_d(C)=0$ or
\[ \dim W^r_d(C) \le h^0(\mathcal{O}_C(2D)) - 3r \le \begin{cases} d-3r+1 & \mbox{if} \ d \le g \\ 2d - 3r - g + 1 & \mbox{if} \ d \ge g \end{cases}. \]
where $J(C)$ denotes the Jacobian variety of $C$.
\end{thm}

We will also use the following lemmas that are a simple application of the dimension estimate of multiples of the hyperplane linear system on a curve of degree $d$ in $\PP^r$; cf.  \cite[p.115]{ACGH} or \cite[Chapter 3.a]{Hr}.

\begin{lem}
\label{castelnuovo}
Let $r \ge 3$ and let $C$ be a smooth irreducible non-degenerate curve of degree $d$ and genus $g$ in $\PP^r$. Then
\[ r \le \begin{cases} \frac{d+1}{3} & \mbox{if} \ d \le g \\ \frac{1}{3}(2d-g+1) & \mbox{if} \ d \ge g \end{cases}. \]
\end{lem}
\begin{proof}
Set $m = \lfloor \frac{d-1}{r-1} \rfloor$. Suppose $d \le g$ and assume that $r > \frac{d+1}{3}$, so that $1 \le m \le 3$. If $m = 1$ we have $g \le \pi (d,r) = d - r \le g - r$, a contradiction. If $m = 2$ we get $g \le \pi (d,r) = 2d - 3r + 1 < d \le g$, a contradiction. Then $m = 3$ and $g \le \pi (d,r) = 3d - 6r + 3 \le d - 1 < g$, again a contradiction. Therefore $r \le \frac{d+1}{3}$ when $d \le g$.

Suppose now  $d \ge g$, so that $2d > 2g-2$ and $H^1(\mathcal{O}_C(2)) = 0$. By \cite[p.115]{ACGH} we have
$h^0(\mathcal{O}_C(2)) \geq 3r$, whence, by Riemann-Roch, $2d-g +1 \geq 3r$, that is $r \le \frac{1}{3}(2d-g+1)$.
\end{proof}

In the next result we will use the second and third Castelnuovo bounds $\pi_1(d,r)$ and $\pi_2(d,r)$.

\begin{lem}
\label{castelnuovo2}
Let $r \ge 4$ and let $C$ be a smooth irreducible non-degenerate curve of degree $d$ and genus $g \ge 2$ in $\PP^r$. 
Assume that either 
\begin{itemize}
\item[(i)] $d \ge 2r+1$ and $g > \pi_1(d,r)$ 

or 
\item[(ii)] $C$ is linearly normal, $r \ge 8, d \ge 2r+3$ and either $g > \pi_2(d,r)$ or $g = \pi_1(d,r)$.
\end{itemize} 
Then $C$ admits a degeneration $\{C_t \subset \PP^r \}_{t \in \PP^1}$ to a singular stable curve.
\end{lem}
\begin{proof}
Under hypothesis (i) it follows by \cite[Theorem 3.15]{Hr} that $C$ lies on a surface of degree $r-1$ in $\PP^r$. Under hypothesis (ii) it follows by
\cite[Theorem 2.10]{P} or \cite[Theorem 3.15]{Hr} respectively, that $C$ lies on a surface of degree $r-1$ or $r$ in $\PP^r$.

To do the case of the surface of degree $r-1$, we recall the following notation (see \cite[\S 5.2]{Ht}). Let $e \ge 0$ be an integer and let $\mathcal{E} = \mathcal{O}_{\PP^1} \oplus \mathcal{O}_{\PP^1}(-e)$. On the ruled surface $X_e = \PP \mathcal{E}$, let $C_0$ be a section in $|\mathcal{O}_{\PP \mathcal{E}}(1)|$ and let $f$ be a fiber. Then any curve $D \sim aC_0 + bf$ has arithmetic genus $\frac{1}{2}(a-1)(2b - ae - 2)$. Write $r-1=2n-e$ for some $n \ge e$ and let $S_{n,e}$ be the image of $X_e$ under the linear system $|C_0 + nf|$. This linear system embeds $X_e$ when $n > e$, while, when $n=e$, it contracts $C_0$ to a point and is an isomorphism elsewhere, thus $S_{e,e}$ is a cone. As is well-known (see for example \cite[Proposition 3.10]{Hr}), every irreducible surface of degree $r-1$ in $\PP^r$ is either the Veronese surface $v_2(\PP^2) \subset \PP^5$ or an $S_{n,e}$. 

If $C \subset v_2(\PP^2)$ then $C \sim a Q$ where $Q$ is a conic and $d = 2a \ge 11$, so that $a \ge 6$. Let $C_1$ be general in $|(a-1)Q|$ and let $C_2$ be general in $|Q|$. Now $C_1$ and $C_2$ are smooth irreducible, $C_1 \cdot C_2 = a-1 \ge 5$ and they intersect transversally. Therefore $C$ specializes to the singular stable curve $C_1 \cup C_2$. 

Now suppose that $C \subset S_{n,e}$. 

We deal first with the case $n > e$. 

We have $C \sim aC_0 + bf$ for some integers $a, b$ such that $a \ge 2$ (because $g \ge 2$) and, as $C$ is smooth irreducible, we get by \cite[Corollary V.2.18(b)]{Ht} that either $e=0, b \ge 2$ (if $b = 1$ then $g = 0$) or $e > 0, b \ge ae$. Consider first the case $a \ge 3$. If $e=0$ or if $e > 0$ and $b > ae$, let $C_1$ be general in $|C - f|$ and $C_2$ be a general fiber. As above $C_1$ and $C_2$ are smooth irreducible, intersect transversally and $C_1 \cdot C_2 = a \ge 3$. Therefore $C_1 \cup C_2$ is a singular stable curve and $C$ specializes to it. If $e > 0$ and $b = ae$ let $C_1$ be general in $|C_0 + ef|$ and $C_2$ be general in $|(a-1)(C_0 + ef)|$. As above both $C_1$ and $C_2$ are smooth irreducible, intersect transversally and $C_1 \cdot C_2 = e (a-1) \ge 3$ unless $a=3, e=1, b=3$, which is not possible since then $g=1$. Therefore $C_1 \cup C_2$ is a singular stable curve and $C$ specializes to it. Now suppose that $a = 2$. Then $g = b-e-1 \ge 2$ and therefore $b \ge e+3$. Let $C_1 = C_0$ and let $C_2$ be general in $|C_0 + bf|$. As above $C_1$ and $C_2$ are smooth irreducible, intersect transversally and $C_1 \cdot C_2 = b-e \ge 3$. Therefore $C_1 \cup C_2$ is a singular stable curve and $C$ specializes to it. This concludes the case $n > e$. 

If $n = e$ let $\widetilde{C}$ be the strict transform of $C$ on $X_e$. Then $C \cong \widetilde{C} \sim aC_0 + bf$ for some integers $a, b$. We get $d = \widetilde{C} \cdot (C_0 + ef) = b$ and, as $C$ is smooth, $d-ae = \widetilde{C} \cdot C_0 = 0, 1$. Since $e = r-1$, setting $\eta = 0, 1$ we get $d = a(r-1) + \eta$. Note that if $a \le 2$ we have $d \le 2r-1$, a contradiction. Hence $a \ge 3$. If $\eta = 1$ let $\widetilde{C}_1$ be general in $|\widetilde{C} - f|$ and $\widetilde{C}_2$ be a general fiber. As above $\widetilde{C}_1$ and $\widetilde{C}_2$ are smooth irreducible, intersect transversally and $\widetilde{C}_1 \cdot \widetilde{C}_2 = a \ge 3$. Therefore $\widetilde{C}_1 \cup \widetilde{C}_2$ is a singular stable curve and $\widetilde{C}$ specializes to it. On the other hand $\widetilde{C}, \widetilde{C}_1$ and $\widetilde{C}_2$ get mapped isomorphically in $\PP^r$, therefore also $C$ specializes to a singular stable curve. If $\eta = 0$ let $\widetilde{C}_1$ be general in $|C_0 + ef|$ and $\widetilde{C}_2$ be general in $|(a-1)(C_0 + ef)|$. Again $\widetilde{C}_1$ and $\widetilde{C}_2$ are smooth irreducible, intersect transversally, $\widetilde{C}_1 \cdot \widetilde{C}_2 = (a-1)(r-1) \ge 6$ and they get mapped isomorphically in $\PP^r$, therefore $C$ specializes to a singular stable curve image of $\widetilde{C}_1 \cup \widetilde{C}_2$.

This concludes the case of the surface of degree $r-1$.

We now consider the case $r \ge 8, d \ge 2r+3$, $C$ is linearly normal and lies on a surface of degree $r$ in $\PP^r$. By a classical theorem of del Pezzo and Nagata (see \cite[Thorem 8]{N}) we have that such a surface is either a cone over an elliptic normal curve in $\PP^{r-1}$ or the 3-Veronese surface $v_3(\PP^2) \subset \PP^9$ or the image of $X_e, e = 0, 1, 2$ with the linear system
$|2C_0 + 2f|, |2C_0 + 3f|$ or $|2C_0 + 4f|$ respectively. The case $v_3(\PP^2)$ is done exactly as the case $v_2(\PP^2)$ above, while the cases $e=0,1$ are done exactly as the case $S_{n,e}, n>e$ above, since the linear systems $|2C_0 + 2f|, |2C_0 + 3f|$ are very ample and therefore a degeneration on $X_e$ gives a degeneration in $\PP^8$. In the case $e = 2$ let $\widetilde{C}$ be the strict transform of $C$ on $X_2$. Then $C \cong \widetilde{C} \sim aC_0 + bf$ for some integers $a, b$. We get $d = \widetilde{C} \cdot (2C_0 + 4f) = 2b$ and, as $C$ is smooth, $b-2a = \widetilde{C} \cdot C_0 = \eta$. Also if $a \le 2$ we have $d \le 10$, a contradiction. Hence $a \ge 3$. Now exactly as in the case $n=e$ above we conclude that 
$C$ specializes to a singular stable curve.

It remains to do the case when $C$ is contained in the cone over an elliptic normal curve in $\PP^{r-1}$. Let $E \subset \PP^{r-1}$ be a linearly normal smooth irreducible elliptic curve of degree $r$, set $\mathcal{E} = \mathcal{O}_E \oplus  \mathcal{O}_E(-1)$ and let $\pi : \PP \mathcal{E} \to E$ be the standard map. By \cite[Example V.2.11.4]{Ht} the cone is the image of $\PP \mathcal{E}$ under the linear system $|C_0 + \pi^{\ast} \mathcal{O}_E(1)|$, which contracts $C_0$ to the vertex and is an isomorphism elsewhere. In particular it follows that $C_0 + \pi^{\ast} \mathcal{O}_E(1)$ is big and base-point-free. Let $\widetilde{C}$ be the strict transform of $C$ on $\PP \mathcal{E}$, so that $C \cong \widetilde{C} \sim aC_0 + \pi^{\ast} M$ for some integer $a$ and some line bundle $M$ on $E$ of degree $b$. As before we have $d = \widetilde{C} \cdot (C_0 + rf) = b$ and, as $C$ is smooth, $d-ar = \widetilde{C} \cdot C_0 = \eta$, so that $a \ge 3$. 

Assume that $\eta = 0$. We claim that $M \cong \mathcal{O}_E(a)$. In fact if $M \not \cong \mathcal{O}_E(a)$ we compute
\[ h^0(\PP \mathcal{E}, aC_0 + \pi^{\ast} M) = h^0(E, \pi_{\ast}(aC_0 + \pi^{\ast} M)) = \]
\[ = h^0((\Sym^a\mathcal{E}) \otimes M) = h^0(\bigoplus\limits_{i=0}^a M(-i)) = \sum\limits_{i=0}^{a-1} (a-i) r \]
while
\[ h^0(\PP \mathcal{E}, (a-1)C_0 + \pi^{\ast} M) = h^0(\bigoplus\limits_{i=0}^{a-1} M(-i)) = \sum\limits_{i=0}^{a-1} (a-i) r \]
and therefore the linear system $|aC_0 + \pi^{\ast} M|$ has $C_0$ as base component. But $\widetilde{C}$ is irreducible and $\widetilde{C} \neq C_0$, whence a contradiction. Hence $M \cong \mathcal{O}_E(a)$ and $\widetilde{C} \sim a(C_0 +  \pi^{\ast} \mathcal{O}_E(1))$. Let $\widetilde{C}_1$ be general in $|C_0 +  \pi^{\ast} \mathcal{O}_E(1)|$ and let $\widetilde{C}_2$ be general in $|(a-1)(C_0 +  \pi^{\ast} \mathcal{O}_E(1))|$. Now $\widetilde{C}_1$ and $\widetilde{C}_2$ are smooth irreducible by Bertini's theorem, intersect transversally and $\widetilde{C}_1 \cdot \widetilde{C}_2 = (a-1) r \ge 16$. Therefore $\widetilde{C}_1 \cup \widetilde{C}_2$ is a singular stable curve and $\widetilde{C}$ specializes to it. On the other hand $\widetilde{C}, \widetilde{C}_1$ and $\widetilde{C}_2$ get mapped isomorphically in $\PP^r$, therefore also $C$ specializes to a singular stable curve.

Finally let us do the case $\eta = 1$. Then $M(-a)$ has degree $1$ and thefore there is a point $P \in E$ such that $M \cong \mathcal{O}_E(a)(P)$. Let $F = \pi^{\ast}(P)$ be a fiber and let $\widetilde{C}_1$ be general in $|\widetilde{C} - F|$. Note that $\widetilde{C} - F \sim a(C_0 +  \pi^{\ast} \mathcal{O}_E(1))$ and therefore $\widetilde{C}_1$ is smooth irreducible. Again $\widetilde{C}_1$ and $\widetilde{C}_2$ intersect transversally and $\widetilde{C}_1 \cdot \widetilde{C}_2 = a \ge 3$. Hence $\widetilde{C}_1 \cup \widetilde{C}_2$ is a singular stable curve and $\widetilde{C}$ specializes to it. Also $\widetilde{C}, \widetilde{C}_1$ and $\widetilde{C}_2$ get mapped isomorphically in $\PP^r$, therefore also $C$ specializes to a singular stable curve.
\end{proof}

\section{Irreducibility of $\mathcal{H}_{g+1,g,3}$ for small genus $g$}
 
The irreducibility of  $\mathcal{H}_{g+1,g,3}$ has been known for $g\ge 9$; cf.  \cite[Theorem 2.6 and Theorem 2.7]{K}. In this section we prove that  any non-empty $\mathcal{H}_{g+1,g,3}$ is irreducible of expected dimension for $g\le 8$, whence for all $g$ without any restriction on the genus $g$. 

\begin{prop}
\label{ire} 
$\mathcal{H}_{g+1,g,3}$ is irreducible of expected dimension $4(g+1)$ if $g\ge 6$ and is empty if $g\le 5$. Moreover, $\dim\pi(\mathcal{H}_{g+1,g,3})=3g-3$ if $g\ge 8$, $\dim\pi(\mathcal{H}_{8,7,3})=17$ and $\dim\pi(\mathcal{H}_{7,6,3})=13$.

\end{prop}
\begin{proof} By the Castelnuovo genus bound, one can easily see that there is no smooth non-degenerate curve in $\PP^3$ of degree $g+1$ and genus $g$ if $g\le 5$. Hence $\mathcal{H}_{g+1,g,3}$ is empty for $g\le 5$. We now treat separately the other cases. 

\noindent (i) $g=6$: A smooth curve $C$ of genus $6$ with a very ample $g^3_7$ is trigonal; $|K-g^3_7|=g^1_3$. Furthermore,  $C$ has a unique trigonal pencil  by Castelnuovo-Severi inequality and the $g^3_7$ is unique as well. Conversely a trigonal curve of genus $6$ has a unique trigonal pencil and the residual series $g^3_7=|K-g^1_3|$ is very ample which is the unique $g^3_7$. Hence $\mathcal{G}\subset\mathcal{G}^3_7$ is birationally equivalent to the irreducible locus of trigonal curves $\mathcal{M}^1_{g,3}$. Therefore it follows that $\mathcal{H}_{7,6,3}$ is irreducible which is a $\PP GL(4)$-bundle over  the irreducible  locus $\mathcal{M}^1_{g,3}$ and $\dim\mathcal{H}_{7,6,3}=\dim\mathcal{M}^1_{g,3}+\dim \PP GL(4)=(2g+1)+15=28$ . 

\noindent (ii) $g=7$: First we note that a smooth curve $C$ of degree $8$ in $\PP^3$ of genus $7$ does not lie on a quadric surface; there is no integer solution to the equation $a+b=8$,  $(a-1)(b-1)=7$ assuming $C$ is of type $(a,b)$ on a quadric surface. We then claim that $C$ is residual to a line in a complete intersection of two cubic surfaces; from the exact sequence
 $0\rightarrow \mathcal{I}_C(3)\rightarrow\mathcal{O}_{\PP^3}(3)\rightarrow\mathcal{O}_C(3)\rightarrow 0$, one sees that $h^0(\PP^3, \mathcal{I}_C(3))\ge 2$ and hence $C$ lies on two irreducible cubics. Note that $\deg C=g+1=8=3\cdot 3-1$ and therefore $C$ is a curve residual to a line in a complete intersection of two cubics, that is  $C\cup L= X$ where L is a line and $X$ is a complete intersection of two cubics.  Upon fixing a line $L\subset\PP^3$,  we consider the linear system $\mathcal{D}=\PP(H^0(\PP^3, \mathcal{I}_L(3)))$ consisting of cubics containing the line $L$. Note that any $4$ given points on $L$ impose independent conditions on cubics and hence $\dim\mathcal{D}=\dim\PP(H^0(\PP^3,\mathcal{O}(3)))-4=19-4=15$. Since our curve $C$ is completely determined by a pencil of cubics containing a line $L\subset \PP^3$, we see that $\mathcal{H}_{8,7,3}$ is a $\GG(1,15)$ bundle over $\GG(1,3)$, the space of lines in $\PP^3$. Hence $\mathcal{H}_{8,7,3}$ is irreducible of dimension $\dim \GG(1,15)+\dim\GG(1,3)=28+4=32=4\cdot 8$. By taking the residual series $|K_C-g^3_8|=g^1_4$ of a very ample $g^3_8$, we see that $\mathcal{H}_{8,7,3}$ maps into the irreducible closed locus $\mathcal{M}^1_{g,4}$  consisting $4$-gonal curves, which is of dimension $2g+3$. We also note that $\dim W^3_8(C)=\dim W^1_4(C)=0$. For if $\dim W^1_4(C) \ge 1$, then $C$ is either trigonal, bielliptic or a smooth plane quintic by Mumford's theorem; cf. \cite[p.193]{ACGH}. Because $g=7$, $C$ cannot be a smooth plane quintic. If $C$ is trigonal with the trigonal pencil $g^1_3$, one may deduce that  $|g^1_3+g^1_5|$ is our very ample $g^3_8$ by the base-point-free pencil trick \cite[p.126]{ACGH};   $\ker\nu\cong H^0(C, \mathcal{F}\otimes\mathcal{L}^{-1})$ where 
$\nu :H^0(C, \mathcal{F})\otimes H^0(C, \mathcal{L})\rightarrow H^0(C, \mathcal{F}\otimes\mathcal{L})$  is the natural cup-product map with $\mathcal{F}=g^3_8$, $\mathcal{L}=g^1_3$ and $\mathcal{F}\otimes\mathcal{L}^{-1}$ turns out to be a $g^1_5$. 
Therefore it follows that $C$ is a smooth curve of type $(3,5)$ on a smooth quadric in $\PP^3$. However, we have already ruled out the possibility for $C$ lying on a quadric. If $C$ is  bi-elliptic with a two sheeted map $\phi : C \rightarrow E$ onto an elliptic curve $E$, one sees that $|K-g^3_8|=g^1_4=|\phi^*(p+q)|$ by Castelnuovo-Severi inequality and hence $g^3_8=|K-\phi^*(p+q)|$ where $p, q\in E$. Therefore for any $r\in E$, we have $|g^3_8-\phi^*(r)|=|K-\phi^*(p+q+r)|=|K-g^2_6|=g^2_6$ 
 whereas $g^3_8$ is very ample, a contradiction. Furthermore we see that $\mathcal{H}_{8,7,3}$ dominates the locus $\mathcal{M}^1_{g,4}$, for otherwise the inequality $\dim\pi (\mathcal{H}_{8,7,3})<\dim\mathcal{M}^1_{4,g}=2g+3$ which would lead to the inequality $$\hskip 24pt\dim\mathcal{H}_{8,7,3}=32\le\dim \PP GL(4)+\dim W^3_8(C)+\dim\pi (\mathcal{H}_{8,7,3})<15+17,$$ which is an absurdity. 

\noindent (iii) $g=8$:  Since we have the non-negative Brill-Noether number $\rho (d,g,3)=\rho (9,8,3)=0$, there exists the principal component of $\mathcal{H}_{9,8,3}$ dominating $\mathcal{M}_8$ by Proposition \ref{principal}. Because  almost the same argument as in the proof of  \cite[Theorem 2.6 ]{K} works for this case, we provide only the essential ingredient 
and important issue
adopted for our case $g=8$.   Indeed, the crucial step in the proof of \cite[Theorem 2.6 ]{K} was \cite[Lemma 2.4 ]{K} (for a given $g\ge 9$) in which the author used a rather strong result 
\cite[Lemma 2.3 ]{K}; e.g. if  $\dim W^1_5(C)=1$ on a fixed curve $C$ of genus $g=9$ then $\dim W^1_4(C)=0$. However a similar statement for $g=8$ was not known at that time. In other words, it was not clear at all that  the condition $\dim W^1_5(C)=1$ would imply $\dim W^1_4(C)=0$ for a curve $C$ of genus $g=8$. However by the results of Mukai \cite{M} and Ballico et al. \cite[Theorem 1]{B} it has been shown that the above statement holds for a curve of genus $8$. Therefore the same proof as in \cite[Lemma 2.4, Theorem 2.6]{K} works (even without changing any paragraphs or notation therein). The authors apologize for not being kind enough to provide a full proof; otherwise  this article may become unnecessarily lengthy and tedious. \qedhere
\end{proof}

\section{Non-existence of components of $\mathcal{H}_{d,g,3}$ rigid in moduli}

In this section, we give a strictly positive lower bound for the dimension of the image $\pi (\mathcal{Z})$ of an irreducible component $\mathcal{Z}$ of the Hilbert scheme  $\mathcal{H}_{d,g,3}$ under the natural map 
$\pi: \mathcal{H}_{d,g,r} \dashrightarrow \mathcal{M}_g$,  which will in turn imply  that $\mathcal{H}_{d,g,3}$ has no components rigid in moduli. The non-existence of a such component of $\mathcal{H}_{d,g,3}$ has certainly been known to some people (e.g. cf.  \cite[p.3487]{L2}). However the authors could not find an adequate source of a proof in any literature.  
\vskip 4pt

We start with the following fact about the irreducibility of $\mathcal{H}_{d,g,3}$ which has been proved by Ein \cite[Theorem 4]{E} and Keem-Kim \cite[Theorems 1.5 and 2.6]{KK}. 

\begin{thm}
\label{irredrange}
$\mathcal{H}_{d,g,3}$ is irreducible for $d \ge g+3$ and for $d = g+2, g \ge 5$.
\end{thm}

Using Proposition \ref{principal} and Theorem \ref{irredrange}, one can prove the following rather elementary facts, well known to experts and included for self-containedness, when the genus or the degree of the curves under consideration is relatively low.

\begin{prop}
\label{lowgenus}
For $1\le g \le4$, every non-empty $\mathcal{H}_{d,g,3}$ is irreducible of dimension $\lambda(d,g,3)=4d$. Moreover,  $\mathcal{H}_{d,g,3}$ dominates $\mathcal{M}_g$.
\end{prop}

\begin{proof}
For $d\ge g+3$, we have $\rho(d,g,3)=g-4(g-d+3)\ge 0$ and hence there exists a principal component $\mathcal{H}_0$ which dominates $\mathcal{M}_g$ by Proposition \ref{principal}. Since $\mathcal{H}_{d,g,3}$ is irreducible for $d\ge g+3$ by Theorem \ref{irredrange}, it follows that $\mathcal{H}_{d,g,3}=\mathcal{H}_0$ dominates $\mathcal{M}_g$. Therefore it suffices to prove the statement when $d\le g+2$. 

If $1\le g\le 3$, one has $\pi (d,3)<g$  for $d\le g+2$ and hence $\mathcal{H}_{d,g,3}=\emptyset$.

If $g=4$ we have $d\le6$. Since $\pi(d,3)\le2$ for $d\le5$, one has $\mathcal{H}_{d,4,3}=\emptyset$ for $d\le 5$ and hence we just need to consider $\mathcal{H}_{6,4,3}$. 
\noindent
We note that a smooth curve in $\PP^3$ of degree $6$ and genus $4$ is a canonical curve, that is a curve embedded by the
canonical linear series and vice versa. Hence $\mathcal{G}\subset\mathcal{G}^3_6$ is birationally equivalent to the irreducible variety $\mathcal{M}_4$  and it follows that $\mathcal{H}_{6,4,3}$  is irreducible  which is  a 
$\PP GL(4)$-bundle over an open subset of $\mathcal{M}_4$ or $\mathcal{G}$.  \qedhere
\end{proof}

\begin{prop}
\label{lowdegree}
The Hilbert schemes  $\mathcal{H}_{8,8,3}, \mathcal{H}_{8,9,3}$ and $\mathcal{H}_{9,g,3}$ for $g = 9, 12$ are irreducible, while 
$\mathcal{H}_{9,11,3}$ is empty and $\mathcal{H}_{9,10,3}$ has two irreducible components. 

Moreover, under the natural map $\pi : \mathcal{H}_{d,g,3} \dashrightarrow \mathcal{M}_{g}$, we have
\begin{enumerate}
\item[(i)] $\dim \pi(\mathcal{H}_{8,8,3}) = 17$; 
\item[(ii)] $\dim \pi(\mathcal{H}_{8,9,3}) = 18$;
\item[(iii)] $\dim \pi(\mathcal{Z}) = 21$ both when $\mathcal{Z} = \mathcal{H}_{9,9,3}$ and when $\mathcal{Z}$ is one of the two irreducible components of $\mathcal{H}_{9,10,3}$;
\item[(iv)] $\dim \pi(\mathcal{H}_{9,12,3}) = 23$.
\end{enumerate}
\end{prop}

\begin{proof}
To see that $\mathcal{H}_{9,11,3} = \emptyset$ we use \cite[Corollary 3.14]{Hr}. In fact note that there is no pair of integers $a \ge b \ge 0$ such that $a+b=9, (a-1)(b-1) = 11$. Since the second Castelnuovo bound $\pi_1(9,3) = 10$ and $\pi(9,3) = 12$, it follows that $\mathcal{H}_{9,11,3} = \emptyset$.

As for the other cases, we start with a few general remarks. We will first prove that the Hilbert schemes $\mathcal{H}_{d,g,3}$ or the components $\mathcal{Z} \subset \mathcal{H}_{d,g,3}$ to be considered are irreducible, generically smooth and that their general point represents a smooth irreducible non-degenerate linearly normal curve $C \subset \PP^3$. Moreover we will show that the standard multiplication map 
\[\mu_0: H^0(\mathcal{O}_C(1)) \otimes H^0(\omega_C(-1)) \to H^0(\omega_C) \] 
is surjective. From the above it will then follow, by well-known facts about the Kodaira-Spencer map (see e.g. \cite[Proof of Proposition 3.3]{S}, \cite[Proof of Theorem 1.2]{L1}, that if $N_C$ is the normal bundle of $C$, then 
\begin{equation}
\label{image}
\dim \pi(\mathcal{Z}) = 3g-3+\rho+h^1(N_C) = 4d-15+h^1(N_C)
\end{equation}
and this will give the results in (i)-(iv).

In general, given a smooth surface $S \subset \PP^3$ containing $C$, we have the commutative diagram
\begin{equation}
\label{dia}
\xymatrix{ H^0(\mathcal{O}_S(1))  \otimes H^0(\omega_S(C)(-1)) \ar[r]^{\hskip 1.3cm \nu} \ar[d] & H^0(\omega_S(C))  \ar[d] \\ 
H^0(\mathcal{O}_C(1)) \otimes H^0(\omega_C(-1))  \ar[r]^{\hskip 1cm \mu_0} & H^0(\omega_C) \ar[d] \\ & H^1(\omega_S) = 0 }
\end{equation}
so that $\mu_0$ is surjective when $\nu$ is.

Now consider a smooth irreducible curve $C$ of type $(a,b)$, with $a \ge b \ge 3$, on a smooth quadric surface $Q \subset \PP^3$. In the exact sequence
\[ 0 \to N_{C/Q} \to N_C \to N_{{Q}_{|C}} \to 0 \]
we have that $H^1(N_{C/Q}) = 0$ since $C^2 = 2g-2+2d > 2g-2$ and $N_{{Q}_{|C}} \cong \mathcal{O}_C(2)$, so that
\begin{equation}
\label{h1}
h^1(N_C) = h^1(\mathcal{O}_C(2)).
\end{equation}

Moreover we claim that
\begin{equation}
\label{mu0}
C \ \mbox{is linearly normal and} \ \mu_0 \ \mbox{is surjective}.
\end{equation}
In fact from the exact sequence
\[ 0 \to \mathcal{O}_Q(1-a,1-b) \to \mathcal{O}_Q(1) \to \mathcal{O}_C(1) \to 0 \]
and the fact that $H^i(\mathcal{O}_Q(1-a,1-b))=0$ for $i=0,1$, we see that $h^0(\mathcal{O}_C(1)) = h^0(\mathcal{O}_Q(1)) = 4$. 
Setting $S = Q$ in \eqref{dia} we find that $\nu$ is the surjective multiplication map of bihomogeneous polynomials 
\[ H^0(\mathcal{O}_Q(1,1)) \otimes H^0(\mathcal{O}_Q(a-3,b-3)) \to H^0(\mathcal{O}_Q(a-2,b-2)) \]
on $\PP^1 \times \PP^1$. This proves \eqref{mu0}.

To see (i) note that, since $\pi_1(8,3) = 7$ and $\pi(8,3) = 9$, it follows by \cite[Corollary 3.14]{Hr} that $\mathcal{H}_{8,8,3}$ is irreducible of dimension $32$ and its general point represents a curve of type $(5,3)$ on a smooth quadric. Moreover $H^1(\mathcal{O}_C(2)) = 0$ since $2d = 16 > 2g-2= 14$ and therefore $H^1(N_C)=0$ by \eqref{h1} and $\mathcal{H}_{8,8,3}$ is smooth at the point representing $C$. Now \eqref{mu0} and \eqref{image} give (i).

To see (ii) observe that, by \cite[Example (10.4)]{CS}, $\mathcal{H}_{8,9,3}$ is smooth irreducible of dimension $33$ and its general point represents a curve of type $(4,4)$ on a smooth quadric. Moreover $h^1(\mathcal{O}_C(2)) = h^1(\omega_C) = 1$ and therefore $h^1(N_C)=1$ by \eqref{h1}. Hence \eqref{mu0} and \eqref{image} give (ii).

Finally, to prove (iii) and (iv), consider $\mathcal{H}_{9,g,3}$ for $g = 9, 10$ or $12$. 

By \cite[Theorem 5.2.1]{D} we know that $\mathcal{H}_{9,9,3}$ is irreducible of dimension $36$ and its general point represents a curve $C$ residual of a twisted cubic $D$ in the complete intersection of a smooth cubic $S$ and a quartic $T$.  In the exact sequence
\[ 0 \to N_{C/S} \to N_C \to N_{{S}_{|C}} \to 0 \]
we have that $H^1(N_{C/S}) = 0$ since $C^2 = 25 > 2g-2= 16$ and $N_{{S}_{|C}} \cong \mathcal{O}_C(3)$ that has degree $27$, so that again $H^1(N_{{S}_{|C}}) = 0$ and therefore also $H^1(N_C) = 0$. Hence $\mathcal{H}_{9,9,3}$ is smooth at the point representing $C$. Moreover,  as $D$ is projectively normal, $C$ is also projectively normal. It remains to prove that $\mu_0$ is surjective, whence, by \eqref{dia}, that $\nu$ is surjective. To this end observe that, if $H$ is the hyperplane divisor of $S$, then $K_S + C - H \sim 2H - D$. A general element $D' \in |2H-D|$ is again a twisted cubic and we get the commutative diagram 
\[ \xymatrix{ 0 \ar[d] & 0 \ar[d] \\ H^0(\mathcal{O}_S(H))  \otimes H^0(\mathcal{O}_S) \ar[r]^{\hskip .8cm \nu_1} \ar[d] & H^0(\mathcal{O}_S(H))  \ar[d] \\ H^0(\mathcal{O}_S(H))  \otimes H^0(\mathcal{O}_S(D')) \ar[r]^{\hskip .7cm \nu} \ar[d] & H^0(\mathcal{O}_S(H+D'))  \ar[d] \\ 
H^0(\mathcal{O}_{D'}(H)) \otimes H^0(\mathcal{O}_{D'}(D')) \ar[r]^{\hskip .7cm \tau}  \ar[d] & H^0(\mathcal{O}_{D'}(H+D')) \ar[d] \\ H^0(\mathcal{O}_S(H))  \otimes H^1(\mathcal{O}_S) = 0 & H^1(\mathcal{O}_S(H)) = 0 } \]
from which we see that $\nu$ is surjective since $\nu_1$ is and so is $\tau$, being the standard multiplication map 
\[ H^0(\mathcal{O}_{\PP^1}(3)) \otimes H^0(\mathcal{O}_{\PP^1}(1)) \to H^0(\mathcal{O}_{\PP^1}(4)). \]
Now \eqref{image} gives (iii) for $g = 9$.

In the case $g=10$ it follows by \cite[Example (10.4)]{CS}, that $\mathcal{H}_{9,10,3}$ has two generically smooth irreducible components $\mathcal{Z}_1$ and $\mathcal{Z}_2$, both of dimension $36$, and their general point represents a curve $C$ of type $(6,3)$ on a smooth quadric for $\mathcal{Z}_1$ and a complete intersection of two cubics for $\mathcal{Z}_2$. In the first case, from the exact sequence
\[ 0 \to \mathcal{O}_Q(-4,-1) \to \mathcal{O}_Q(2) \to \mathcal{O}_C(2) \to 0 \]
and the fact that $H^1(\mathcal{O}_Q(2)) = H^2(\mathcal{O}_Q(-4,-1)) = 0$, we get $h^1(\mathcal{O}_C(2)) = 0$ and therefore $h^1(N_C)=0$ by \eqref{h1}. Then \eqref{mu0} and \eqref{image} give (iii) for $\mathcal{Z}_1$. As for $\mathcal{Z}_2$, we have that $N_C \cong \mathcal{O}_C(3)^{\oplus 2}$ and $\omega_C \cong \mathcal{O}_C(2)$, whence $H^1(N_C)=0$. Moreover, if $S$ is one of the two cubics containing $C$, we have the diagram
\[ \xymatrix{ H^0(\mathcal{O}_{\PP^3}(1))  \otimes H^0(\mathcal{O}_{\PP^3}(1)) \ar[r]^{\hskip 1cm \nu_1} \ar[d] & H^0(\mathcal{O}_{\PP^3}(2))  \ar[d] ^{\alpha} \\ 
H^0(\mathcal{O}_S(1)) \otimes H^0(\mathcal{O}_S(1))  \ar[r]^{\hskip 1cm \nu} & H^0(\mathcal{O}_S(2)). } \]
As is well known both $\alpha$ and $\nu_1$ are surjective, whence so is $\nu$ and then $\mu_0$ by \eqref{dia}. Therefore \eqref{image} gives (iii) for $\mathcal{Z}_2$. 

Finally, since $\pi(9,3) = 12$, it follows by \cite[Corollary 3.14]{Hr} that $\mathcal{H}_{9,12,3}$ is irreducible of dimension $38$ and its general point represents a curve of type $(5,4)$ on a smooth quadric. Moreover, from the exact sequence
\[ 0 \to \mathcal{O}_Q(-3,-2) \to \mathcal{O}_Q(2) \to \mathcal{O}_C(2) \to 0 \]
and the fact that $H^i(\mathcal{O}_Q(2))=0$ for $i=1,2$, we get $h^1(\mathcal{O}_C(2)) = h^2(\mathcal{O}_Q(-3,-2)) =$ $h^0(\mathcal{O}_Q(1,0)) = 2$ and therefore $h^1(N_C)=2$ by \eqref{h1}. Hence $h^0(N_C)=38$ and $\mathcal{H}_{9,12,3}$ is smooth at the point representing $C$. Now \eqref{mu0} and \eqref{image} give (iv). \qedhere
\end{proof}

We can now prove our result for $r=3$.

\begin{thm}
\label{r=3}
Let $\mathcal{Z}$ be an irreducible component of $\mathcal{H}_{d,g,3}$ and $g \ge 5$. Then, under the natural map $\pi : \mathcal{H}_{d,g,3} \dashrightarrow \mathcal{M}_{g}$, the following possibilities occur:
\begin{enumerate}
\item[(i)] $\mathcal{Z}$ dominates $\mathcal{M}_g$;
\item[(ii)] $\mathcal{Z} = \mathcal{H}_{7,6,3}$ and $\dim \pi(\mathcal{Z}) = 13$;
\item[(iii)] $\mathcal{Z} = \mathcal{H}_{8,7,3}$ or $\mathcal{H}_{8,8,3}$ and $\dim \pi(\mathcal{Z}) = 17$; 
\item[(iv)] $\mathcal{Z} = \mathcal{H}_{8,9,3}$ and $\dim \pi(\mathcal{Z}) = 18$;
\item[(v)] $\mathcal{Z} = \mathcal{H}_{9,9,3}$ or $\mathcal{Z} \subset \mathcal{H}_{9,10,3}$ and $\dim \pi(\mathcal{Z}) = 21$;
\item[(vi)] $\dim \pi(\mathcal{Z}) \ge 23$.
\end{enumerate}
\end{thm}

\begin{proof}
We first make the following general remark, which will also be used in the proof of Theorem \ref{r>3}. 

Let $r \ge 3$, let $\mathcal{Z}$ be an irreducible component of $\mathcal{H}_{d,g,r}$ not dominating $\mathcal{M}_g$ and let $C$ be a smooth irreducible non-degenerate curve of degree $d$ and genus $g$ in $\PP^r$ corresponding to a general point $c \in \mathcal{Z}$. We claim that $\mathcal{O}_C(1)$ is special. 

In fact $\mathcal{Z}$ is a $\PP GL(r+1)$-bundle over an open subset of a component $\mathcal{G}_1$ of $\mathcal{G}$ (see \S 1). If $\mathcal{O}_C(1)$ is non-special then, by Riemann-Roch, $d \ge g + r$ and $\mathcal{G}_1$ must coincide with $\mathcal{G}_0$ of Proposition \ref{principal}. But $\mathcal{G}_0$ dominates $\mathcal{M}_{g}$, so that also $\mathcal{Z}$ does, a contradiction.

Therefore $\mathcal{O}_C(1)$ is special. Set $\alpha = \dim |\mathcal{O}_C(1)|$, so that $\alpha \ge r$. 

We now specialize to the case $r=3$.

First we notice that, in cases (ii)-(v), using Propositions \ref{ire} and \ref{lowdegree}, $\mathcal{H}_{d,g,3}$ is irreducible, except for $\mathcal{H}_{9,10,3}$, and the dimension of the image under $\pi$ of each component is as listed. Also we have $d \geq 7$, for if $d \leq 6$ then $g \le \pi (6,3) = 4$. 

Assume that $\mathcal{Z}$ is not as in (i), (ii) or the first case of (iii) and that $\dim \pi(\mathcal{Z}) \leq 22$. By Theorem \ref{irredrange}, Proposition \ref{ire} and Remark \ref{dom}, we can assume that $d \le g$. 

For any component $\mathcal{G}_1 \subseteq \mathcal{G}\subseteq \mathcal{G}^3_d$, there exists a component $\mathcal{W}$ of $\mathcal{W}^{\alpha}_d$ and a closed subset $\mathcal{W}_1\subseteq \mathcal{W}\subseteq \mathcal{W}^{\alpha}_d$ such that $\mathcal{G}_1$ is a Grassmannian $\mathbb{G}(3,\alpha)$-bundle over a non-empty open subset of $\mathcal{W}_1$. Thus we have
\begin{eqnarray}
\label{2.1}
\lambda(d,g,3)&=&4d\nonumber\\
&\le& \dim \mathcal{Z} \nonumber\\
&\le& \dim\pi(\mathcal{Z})+\dim W^{\alpha}_{d}(C)+\dim \mathbb{G}(3,\alpha)+\dim\mathbb{P}GL(4)\nonumber\\
&\le& \dim W^{\alpha}_{d}(C) + 4 \alpha + 25.
\end{eqnarray}
By Lemma \ref{castelnuovo} we have $\alpha \le \frac{d+1}{3}$.

If $\dim W^{\alpha}_{d}(C) = 0$, \eqref{2.1} gives
\[ 4d \leq \frac{4}{3}(d+1) + 25 \] 
therefore $d \leq 9$. 

If $\dim W^{\alpha}_{d}(C) \geq 1$ then Theorem \ref{agh} implies $\alpha \leq \frac{d}{3}$. By \eqref{2.1} and Theorem \ref{agh} again, we find 
\[ 4d \leq d + \alpha + 26 \leq \frac{4d+78}{3} \] 
that is again $d \le 9$.

If $d = 7$ we find the contradiction $7 \le g \le \pi (7,3) = 6$. If $d = 8$ it follows that $8 \le g \le \pi (8,3) = 9$ and we get that $\mathcal{Z}$ is as in the second case of (iii) or as in case (iv). If $d = 9$ we find that $9 \le g \le \pi (9,3) = 12$. Since, by Proposition \ref{lowdegree}, $\mathcal{H}_{9,11,3}$ is empty and $\dim \pi(\mathcal{H}_{9,12,3}) = 23$, we get case (v).
\end{proof}

\begin{rmk}
\noindent 
\begin{enumerate}[leftmargin=0cm,itemindent=.5cm,labelwidth=\itemindent,labelsep=0cm,align=left]
\item[(i)]  There are no reasons to believe that our estimate on the lower bound of  $\dim \pi(\mathcal{Z})$ is sharp. On the other hand, it would be interesting to have a better estimate (hopefully sharp)  on the lower bound of $\dim \pi(\mathcal{Z})$ and come up with  (irreducible or reducible) examples of Hilbert scheme $\mathcal{H}_{d,g,3}$ with a component $\mathcal{Z}$ 
achieving the bound.
\item[(ii)] If $d \le g^{\frac{2}{3}}$ there is a better lower bound for the dimension of components $\mathcal{Z}$ of $\mathcal{H}_{d,g,3}$ in \cite[Theorem 1.3]{C}. This leads, in this case, to a better lower bound of $\dim \pi(\mathcal{Z})$.
\end{enumerate}
\end{rmk}

Theorem \ref{r=3}  and  Proposition \ref{lowgenus} yield the following immediate corollary.
\begin{cor}
\label{main}
\null \hskip 3cm
\begin{enumerate} 
\item[\rm(i)] $\mathcal{H}_{d,g,3}$ has no component that is rigid in moduli if $g>0$.
\item[\rm(ii)] Let $C \subset \Bbb{P}^r$ be a smooth irreducible and non-degenerate curve of genus $g$ whose only deformations are given by projective transformations. If $g = 0$ or $r \le 3$ then $C$ is a rational normal curve.
\end{enumerate}
\begin{proof} Since (i) is immediate from Theorem \ref{r=3} and Proposition \ref{lowgenus} and since (ii) is trivial for $r=2$, we only need to check (ii) for $g=0$ and $r \ge 3$. Now $C$ belongs to a unique component $\mathcal{H}$ of the Hilbert scheme $\mathcal{H}_{d,0,r}$ with $\dim \mathcal{H} \le (r+1)^2-1$. On the other hand $\dim\mathcal{H} \ge \lambda(d,0,r) = (r+1)d+r-3$, hence $(r+1)^2-1 \ge (r+1)d+r-3$ giving that $d < r+1$. Therefore $d=r$ and $C$ is the rational normal curve.
\end{proof}
\end{cor}

\section{Non-existence of components of $\mathcal{H}_{d,g,r}$ rigid in moduli with $r \ge 4$}
In this section, we  prove the non-existence of  a component of $\mathcal{H}_{d,g,r}$ rigid in moduli in a certain restricted range of $d$, $g>0$ and $r \ge 4$.

\begin{thm}
\label{r>3}
$\mathcal{H}_{d,g,r}$ has no components rigid in moduli if $g > 0$ and
\begin{enumerate}
\item[(i)] $d > \min\{\frac{17g + 72}{64}, \frac{4g + 15}{15}, \max\{\frac{g + 18}{4}, \frac{17g + 44}{64} \}\}$  \ \ \ if $r = 4$;
\item[(ii)] $d > \min\{\frac{9g + 20}{20}, \frac{10g + 17}{22}, \max\{\frac{2g + 25}{5}, \frac{9g + 10}{20} \}\}$ and $d > \frac{g+22}{3}$ for $101 \le g \le 113$, \ \ \ if $r = 5$;
\item[(iii)] $d > \min\{\frac{13g + 20}{22}, \frac{3g + 3}{5}, \max\{\frac{g + 10}{2}, \frac{13g + 10}{22}\}, \max\{\frac{g + 10}{2}, \frac{3g - 1}{5} \}\}$  \ \ \ if $r = 6$;
\item[(iv)] $d > \min\{\frac{19g + 24}{27}, \max\{\frac{4g + 39}{7}, \frac{76g +71}{108} \}\}$  \ \ \ if $r = 7$;
\item[(v)] $d > \min\{\frac{4g + 1}{5}, \frac{5g - 4}{6}\}$  \ \ \ if $r = 8$;
\item[(vi)] $d > \min\{\frac{9g - 5}{10}, \frac{29g + 3}{33}\}$  and $(d,g) \neq (30,34)$ \ \ \ if $r = 9$;
\item[(vii)] $d > \min\{\frac{21g - 4}{22}, \frac{17g + 12}{18}\}$   \ \ \ if $r = 10$;
\item[(viii)] $d > g$  \ \ \ if $r = 11$;
\item[(ix)] $d > \frac{2(r-5)g - r + 14}{r+1}$ \ \ \ if $r \ge 12$.
\end{enumerate}
\end{thm}

\begin{proof}

Suppose that there is a component $\mathcal{Z}$ of $\mathcal{H}_{d,g,r}$ rigid in moduli and let $C$ be a smooth irreducible non-degenerate curve of degree $d$ and genus $g$ in $\PP^r$ corresponding to a general point $c \in \mathcal{Z}$.

Let $\alpha = \dim |\mathcal{O}_C(1)|$, so that $\alpha \ge r$ and note that, as in the proof of Theorem \ref{r=3}, using Proposition \ref{principal}, we have that $\mathcal{O}_C(1)$ is special. In particular $d \le 2g-2$ and $g \ge 2$.

Moreover we claim that $C$ does not admit a degeneration $\{C_t \subset \PP^{\alpha} \}_{t \in \PP^1}$ to a singular stable curve. In fact such a degeneration gives a rational map $\PP^1 \dashrightarrow \overline{\mathcal{M}_g}$ whose image contains two distinct points, namely the points representing $C$ and the singular stable curve. Hence the image must be a curve and therefore the curves in the pencil $\{C_t \subset \PP^{\alpha} \}_{t \in \PP^1}$ cannot be all isomorphic. Now we have a projection $p: \PP^{\alpha} \dashrightarrow \PP^r$ that sends $C \subset \PP^{\alpha}$ isomorphically to $p(C) = C \subset \PP^r$. Thus the pencil gets projected and gives rise to a deformation $p(C_t) \subset \PP^r$ of $C$ in $\mathcal{Z}$ (recall that $C$ represents a general point of $\mathcal{Z}$). For general $t$ we have that $p(C_t)$ is therefore smooth, whence $p(C_t) \cong C_t$. Since $\mathcal{Z}$ is rigid in moduli we get the contradiction $C \cong C_t$ for general $t$. This proves the claim and now, by Lemma \ref{castelnuovo2}, we can and will assume that $g \le \pi_1(d,\alpha)$ when $d \ge 2 \alpha+1$ and that $g \le \pi_2(d,\alpha), g < \pi_1(d,\alpha)$ when $\alpha \ge 8$ and $d \ge 2 \alpha+3$.

Recall again that for any component $\mathcal{G}_1 \subseteq \mathcal{G}\subseteq \mathcal{G}^r_d$, there exists a component $\mathcal{W}$ of $\mathcal{W}^{\alpha}_d$ and a closed subset $\mathcal{W}_1\subseteq \mathcal{W} \subseteq \mathcal{W}^{\alpha}_d$ such that $\mathcal{G}_1$ is a Grassmannian $\mathbb{G}(r,\alpha)$-bundle over a non-empty open subset of $\mathcal{W}_1$. By noting that $\mathcal{W}_1$ is a sub-locus inside $W^{\alpha}_d(C)$ in our current situation,  we come up with an inequality similar to (\ref{2.1}):
\begin{eqnarray}
\label{bound}
\lambda(d,g,r)&=&(r+1)d-(r-3)(g-1)\nonumber\\
&\le& \dim \mathcal{Z}\nonumber\\
&\le& \dim \mathcal{W}_1+\dim \GG(r,\alpha)+\dim\PP GL(r+1)\nonumber\\
&=&\dim \mathcal{W}_1+(r+1)(\alpha-r)+r^2+2r\\ 
&\le& \dim W^{\alpha}_{d}(C)+(r+1)\alpha+r.\nonumber
\end{eqnarray}
This leads to the following four cases.

\noindent CASE 1: $d< g$ and $\dim W^{\alpha}_{d}(C) = 0$.

We have $\alpha \le (d+1)/3$ by Lemma \ref{castelnuovo} and \eqref{bound} gives
\begin{equation}
\label{uno}
d < g, \alpha \le (d+1)/3 \ \mbox{and} \ (r+1)(d-\alpha) - 3 \le (r-3) g.
\end{equation}
\noindent CASE 2: $d< g$ and $\dim W^{\alpha}_{d}(C) \geq 1$. 

By Theorem \ref{agh} we get $\alpha \le d/3$. By (\ref{bound}) and Theorem \ref{agh} again, we find 
\begin{equation}
\label{due}
d < g, \alpha \le d/3 \ \mbox{and} \ rd-(r-2) \alpha - 4 \le (r-3) g.
\end{equation}
\noindent CASE 3: $d \ge g$ and $\dim W^{\alpha}_{d}(C) = 0$. 

We have $\alpha \le (2d-g+1)/3$ by Lemma \ref{castelnuovo} whence, in particular, $d \ge (g+3r-1)/2$. Now \eqref{bound} gives
\begin{equation}
\label{tre}
d \ge g, \alpha \le (2d-g+1)/3 \ \mbox{and} \ (r+1)(d-\alpha) - 3 \le (r-3) g.
\end{equation}
\noindent CASE 4: $d \ge g$ and $\dim W^{\alpha}_{d}(C) \geq 1$.

By Theorem \ref{agh} we have $\alpha \leq (2d-g)/3$ whence, in particular, $d \ge (g+3r)/2$. By (\ref{bound}) and Theorem \ref{agh} again, we find 
\begin{equation}
\label{quattro}
d \ge g, \alpha \le (2d-g)/3 \ \mbox{and} \ (r-1)d-(r-2) \alpha - 4 \le (r-4) g.
\end{equation}
The plan is to show that, given the hypotheses, the inequalities $\eqref{uno}-\eqref{quattro}$ contradict $g \le \pi_1(d,\alpha)$ when $d \ge 2 \alpha+1$ or $g \le \pi_2(d,\alpha), g < \pi_1(d,\alpha)$ when $\alpha \ge 8$ and $d \ge 2 \alpha+3$.

To this end let us observe that $d \ge 2 \alpha + 3$ in cases $\eqref{uno}-\eqref{quattro}$: In fact this is obvious in cases $\eqref{uno}$ and $\eqref{due}$, while in cases $\eqref{tre}$ and $\eqref{quattro}$, using $g \le 2d-3\alpha+1$ and $g \le 2d-3\alpha$ respectively, if $d \le 2 \alpha + 2$, we get $4 \alpha \le 3r-14$ and $4 \alpha \le 2r-10$, both contradicting $\alpha \ge r$. Therefore in the sequel we will always have that $g \le \pi_1(d,\alpha)$ and that either $\alpha \le 7$ or $\alpha \ge 8$ and $g \le \pi_2(d,\alpha), g < \pi_1(d,\alpha)$.

We now recall the notation. Set 
\[ m_1 = \lfloor \frac{d-1}{\alpha} \rfloor, m_2 = \lfloor \frac{d-1}{\alpha+1} \rfloor, \varepsilon_1 = d - m_1 \alpha -1, \varepsilon_2 = d - m_2 (\alpha+1) -1 \] and 
\[ \mu_1 = \begin{cases} 1 & \mbox{if} \ \varepsilon_1 = \alpha - 1 \cr 0 & \mbox{if} \ 0 \le \varepsilon_1 \le \alpha - 2 \cr \end{cases}, \mu_2 = \begin{cases} 2 & \mbox{if} \ \varepsilon_2 = \alpha \cr 1 & \mbox{if} \ \alpha - 2 \le \varepsilon_2 \le \alpha - 1 \cr 0 & \mbox{if} \ 0 \le \varepsilon_2 \le \alpha - 3 \cr\end{cases} \]
so that
\[ \pi_1(d, \alpha) = \binom{m_1}{2} \alpha + m_1 (\varepsilon_1 +1) + \mu_1, \pi_2(d, \alpha) = \binom{m_2}{2} (\alpha +1) + m_2 (\varepsilon_2 +2) + \mu_2. \]
We now deal with the case $r \ge 11$ (and hence $\alpha \ge 11$).

We start with $\eqref{uno}$. If $\alpha \ge d/3$ then either $\alpha = (d+1)/3$ or $\alpha = d/3$. Then $m_2 = 2, \mu_2 = 0$ and $\pi_2(d, \alpha) \le d < g$. Therefore $\alpha \le (d-1)/3$ and $\eqref{uno}$ gives $d \le \frac{3(r-3)g - r+8}{2(r+1)}$, contradicting (viii)-(ix).

Similarly, in $\eqref{due}$, if $\alpha = d/3$ then $m_2 = 2, \mu_2 = 0$ and $\pi_2(d, \alpha) = d - 1 < g$. Therefore $\alpha \le (d-1)/3$ and $\eqref{due}$ gives $d \le \frac{3(r-3)g - r + 14}{2(r+1)}$, contradicting (viii)-(ix).

Now in $\eqref{tre}$, if $\alpha \ge (2d-g)/3$ then either $\alpha = (2d-g)/3$ or $\alpha = (2d-g+1)/3$. We find $m_2 = 2, \mu_2 = 0$ and $\pi_2(d, \alpha) \le g -1$. Therefore $\alpha \le (2d-g-1)/3$ and $\eqref{tre}$ gives $d \le \frac{2(r-5)g - r + 8}{r+1}$, contradicting (viii)-(ix).

Instead in $\eqref{quattro}$, if $\alpha = (2d-g)/3$ we find $m_2 = 2, \mu_2 = 0$ and $\pi_2(d, \alpha) = g -1$. Therefore $\alpha \le (2d-g-1)/3$ and $\eqref{quattro}$ gives $d \le \frac{2(r-5)g - r + 14}{r+1}$, contradicting (viii)-(ix). This concludes the case $r \ge 11$.

Assume now that $4 \le r \le 10$. 

We first claim that \eqref{tre} and \eqref{quattro} do not occur. In fact note that we have 
\begin{equation}
\label{cinque}
d \ge \max\{r+2, g, (g+3r-1)/2\} \ \mbox{in} \ \eqref{tre} \ \mbox{and} \ d \ge \max\{r+2, g, (g+3r)/2\} \ \mbox{in} \ \eqref{quattro}. 
\end{equation}
Plugging in $\alpha \le (2d-g+1)/3$ in \eqref{tre} and $\alpha \le (2d-g)/3$ in \eqref{quattro} we get
\[ d \le \frac{2(r-5)g + r+10}{r+1}  \ \mbox{in case} \ \eqref{tre} \ \mbox{and} \ d \le \frac{2(r-5)g+12}{r+1}  \ \mbox{in case} \ \eqref{quattro} \]
and it is easily seen that these contradict \eqref{cinque}. 

Therefore, in the sequel, we consider only $\eqref{uno}$ and $\eqref{due}$.

If $\alpha \le 7$ (whence $r \le 7$), we see that $\eqref{uno}$ gives 
\begin{equation}
\label{unosette}
d \le \frac{(r-3)g + 7r + 10}{r+1}
\end{equation}
and $\eqref{due}$ gives 
\begin{equation}
\label{duesette}
d \le \frac{(r-3)g + 7r - 10}{r}.
\end{equation}
Now assume $\alpha \ge 8$, so that $g \le \pi_2(d,\alpha), g < \pi_1(d,\alpha)$. Set $i = d + 1 - 3 \alpha$ and $j = d - 3 \alpha$. 

Then $\eqref{uno}$ implies $d < g, i \ge 0$,
\begin{equation}
\label{unobis}
\alpha (m_1-1) [\frac{r-3}{2}m_1 - r - 1] + (\varepsilon_1 + 1) [(r - 3)m_1 - r - 1] + 3 + \mu_1(r - 3) > 0
\end{equation}
and
\begin{equation}
\label{unotris}
(\alpha + 1)(m_2-1) [\frac{r-3}{2}m_2 - r - 1] + (\varepsilon_2 + 1) [(r - 3)m_2 - r - 1 ] - r + 2 + (m_2 + \mu_2)(r - 3) \ge 0.
\end{equation}
On the other hand $\eqref{due}$ implies $d < g, j \ge 0$,
\begin{equation}
\label{duebis}
\alpha [(r - 3) \binom{m_1}{2} - m_1 r + r - 2] + (\varepsilon_1 + 1) [(r - 3)m_1 - r ] + 4 + \mu_1(r - 3) > 0
\end{equation}
and
\begin{equation}
\label{duetris}
(\alpha + 1)[(r - 3) \binom{m_2}{2} - m_2 r + r - 2 ] + (\varepsilon_2 + 1) [(r - 3)m_2 - r ] - r + 6 + (m_2 + \mu_2)(r - 3) \ge 0.
\end{equation}
Suppose now $r=4$. It is easily seen that \eqref{unobis} implies $m_1 \ge 9$ and $i \ge 7 \alpha + 1$, so that $\alpha \le \frac{d}{10}$. Plugging in $\eqref{uno}$ we contradict (i). On the other hand \eqref{duebis} implies $m_1 \ge 8$, so that $\alpha \le \frac{d-1}{8}$, and also $j \ge \frac{11 \alpha - 2}{2}$, so that $\alpha \le \frac{2d+2}{17}$. Moreover \eqref{duetris} implies $m_2 \ge 8$, so that $\alpha \le \frac{d-9}{8}$, and also $j \ge \frac{11 \alpha + 12}{2}$, so that $\alpha \le \frac{2d-12}{17}$. Plugging in $\eqref{due}$ and using \eqref{unosette} and \eqref{duesette} we contradict (i) and the case $r = 4$ is concluded.

If $r=5$ it is easily seen that \eqref{unobis} implies $m_1 \ge 5$ and $i > 3 \alpha + 1$, so that $\alpha < \frac{d}{6}$. Also \eqref{unotris} implies $m_2 \ge 5$, so that $i \ge 3 \alpha + 5$, hence $\alpha \le \frac{d-4}{6}$. Plugging in $\eqref{uno}$ we contradict (ii). On the other hand \eqref{duebis} implies $m_1 \ge 5$, so that $\alpha \le \frac{d-1}{5}$, and also $j \ge \frac{12 \alpha - 4}{5}$, so that $\alpha \le \frac{5d+4}{27}$. Moreover \eqref{duetris} implies $m_2 \ge 5$ and also $j \ge \frac{12 \alpha + 16}{5}$, so that $\alpha \le \frac{5d-16}{27}$. Plugging in $\eqref{due}$ and using \eqref{unosette} and \eqref{duesette} we contradict (ii) and the case $r = 5$ is done.

When $r=6$ we see that \eqref{unobis} implies $m_1 \ge 4$ and $\alpha \le \frac{d-1}{4}$, but also $i > \frac{8 \alpha + 2}{5}$, so that $\alpha < \frac{5d+3}{23}$. Also \eqref{unotris} implies $m_2 \ge 4$, so that $i \ge \frac{ 8 \alpha + 20}{5}$, hence $\alpha \le \frac{5d-15}{23}$. Plugging in $\eqref{uno}$ we contradict (iii). On the other hand \eqref{duebis} implies $m_1 \ge 4$, so that $\alpha < \frac{d-1}{4}$, and also $j > \frac{4 \alpha - 2}{3}$, so that $\alpha < \frac{3d+2}{13}$. Moreover \eqref{duetris} implies $m_2 \ge 4$ so that $\alpha \le \frac{d-5}{4}$, and also $j \ge \frac{4 \alpha + 7}{3}$, so that $\alpha \le \frac{3d-7}{13}$. Plugging in $\eqref{due}$ and using \eqref{unosette} and \eqref{duesette} we contradict (iii) and we have finished the case $r = 6$.

Now assume that $r=7$. We have that \eqref{unobis} implies $m_1 \ge 3$ and $i \ge \alpha + 1$, so that $\alpha \le \frac{d}{4}$. Plugging in $\eqref{uno}$ we contradict (iv). On the other hand \eqref{duebis} implies $m_1 \ge 3$ and $j > \frac{4 \alpha - 4}{5}$, so that $\alpha < \frac{5d+4}{19}$. Moreover \eqref{duetris} implies $m_2 \ge 3$ and $j \ge \frac{4 \alpha + 1}{5}$, so that $\alpha \le \frac{5d-1}{19}$. Plugging in $\eqref{due}$ and using \eqref{unosette} and \eqref{duesette} we contradict (iv). This concludes the case $r = 7$.

If $r=8$ we find from \eqref{unotris} that $m_2 \ge 3$ and $i \ge \frac{\alpha + 6}{2}$, so that $\alpha \le \frac{2d-4}{7}$. Plugging in $\eqref{uno}$ we contradict (v). On the other hand \eqref{duetris} implies $m_2 \ge 3$ so that $\alpha \le \frac{d-4}{3}$ and also $j \ge \frac{3 \alpha + 11}{7}$, so that $\alpha \le \frac{7d-11}{24}$. Plugging in $\eqref{due}$ we contradict (v) and the case $r = 8$ is proved.

Now let $r=9$. Then \eqref{unotris} gives $m_2 \ge 3$ and $i \ge \frac{2\alpha + 23}{8}$, so that $\alpha \le \frac{8d-15}{26}$. Plugging in $\eqref{uno}$ we contradict (vi). On the other hand \eqref{duetris} gives $m_2 \ge 2$ and $j \ge 3$, so that $\alpha \le \frac{d-3}{3}$. We also get $m_2 = 2$ if and only if $(d,g) = (30,33), (30,34)$. Then, if $(d,g) = (30,33), (30,34)$, we see that  \eqref{duetris} implies $m_2 \ge 3$ and $j \ge \frac{2 \alpha + 14}{9}$, so that $\alpha \le \frac{9d-14}{29}$. Plugging in $\eqref{due}$ we contradict (vi) and we are done with the case $r = 9$.

Finally let us do the case $r=10$. We see that \eqref{unotris} gives $m_2 \ge 2$ and $i \ge 4$, so that $\alpha \le \frac{d-3}{3}$. Plugging in $\eqref{uno}$ we contradict (vii). On the other hand \eqref{duebis} gives $m_1 \ge 3$ and $j > \frac{\alpha-4}{11}$, so that $\alpha < \frac{11d+4}{34}$. Also \eqref{duetris} implies $m_2 \ge 2$ and $j \ge 2$, so that $\alpha \le \frac{d-2}{3}$. Plugging in $\eqref{due}$ we contradict (vii) and we are done with the case $r = 10$.
\qedhere
\end{proof}

\section*{Acknowledgements}
We would like to thank the referee for suggesting that we could add (ii) of Corollary \ref{main}.

\bibliographystyle{ams}

\end{document}